\documentclass[12pt]{amsart}

\usepackage{amssymb, amsmath, tikz}
\usetikzlibrary{arrows}

\newtheorem{theorem}{Theorem}[section]
\newtheorem{lemma}[theorem]{Lemma}

\theoremstyle{definition}
\newtheorem{definition}[theorem]{Definition}
\newtheorem{corollary}[theorem]{Corollary}
\newtheorem{proposition}[theorem]{Proposition}
\newtheorem{example}[theorem]{Example}

\numberwithin{equation}{section}

\theoremstyle{remark}

\numberwithin{equation}{section}

\newcommand{\field}{{\bf k}} 
 
\newcommand{\betti}{\tilde{\beta}}

\DeclareMathOperator{\lk}{lk}
\DeclareMathOperator{\rk}{rk}

\title[]{A classification of the face numbers of Buchsbaum simplicial posets}

\author{Jonathan Browder}
\address{Aalto University \\ Department of Mathematics \\ Helsinki, Finland}
\email{jonathan.browder@aalto.fi}
\author{Steven Klee}
\address{Seattle University\\ Department of Mathematics \\ 901 12th Avenue \\ Seattle, WA 98122}
\email{klees@seattleu.edu}
\urladdr{http://http://fac-staff.seattleu.edu/klees/web/}

\date{\today}

\begin{document}

\begin{abstract} The family of Buchsbaum simplicial posets generalizes the family of simplicial cell manifolds.  The $h'$-vector of a simplicial complex or simplicial poset encodes the combinatorial and topological data of its face numbers and the reduced Betti numbers of its geometric realization. Novik and Swartz  showed that the $h'$-vector of a Buchsbaum simplicial poset satisfies certain simple inequalities; in this paper we show that these necessary conditions are in fact sufficient to characterize the $h'$-vectors of Buchsbaum simplicial posets with prescribed Betti numbers.

\end{abstract}

\maketitle



\section{Introduction}

A (finite) \textit{simplicial poset} is a poset with a unique minimal element $\hat{0}$ in which each interval $[\hat{0},\sigma]$ is isomorphic to a Boolean lattice. Of particular interest are the families of Cohen-Macaulay and Buchsbaum simplicial posets.  One may view the former as a generalization of the family of simplicial cell spheres and balls and the latter as a generalization of the family of simplicial cell manifolds (with or without boundary).

One natural invariant of a $(d-1)$-dimensional simplicial poset $P$ is its $f$-vector, $f(P) = (f_{-1}(P), f_0(P), \ldots, f_{d-1}(P))$,  where $f_i(P)$ counts the number of $i$-dimensional faces of $P$.  In recent decades a large amount of effort has gone into characterizing the possible $f$-vectors of various classes of simplicial posets. Stanley \cite{Stanley-CM-posets} provided such a characterization for Cohen-Macaulay simplicial posets; Masuda \cite{MR2139917} characterized the face numbers of simplicial posets that are Gorenstein*. Murai \cite{Murai-simplicial-cell} gave a characterization of the $f$-vectors of simplicial cell posets whose geometric realizations are products of spheres  and later a characterization for simplicial posets whose geometric realizations are balls \cite{MR3003273}. Building on work of Novik and Swartz \cite{Novik-Swartz}, we here characterize the face numbers of all Buchsbaum simplicial posets with prescribed Betti numbers.

When studying the $f$-vector it is often it is more convenient to consider instead the $h$-vector of $P$, $h(P) = (h_0(P), h_1(P) \ldots h_d(P))$,  a transformation of $f(P)$ that contains equivalent information. In the case that $P$ is Cohen-Macaulay, the $h$-numbers of $P$ count the graded dimensions of a certain module associated to $P$; in the more general case that $P$ is Buchsbaum these dimensions are given by the $h'$-numbers of $P$, which are linear combinations of the $h$-numbers and Betti numbers $\widetilde{\beta}_i(P) := \dim_{\field}\widetilde{H}_i(P;\field)$ of $P$.

Stanley \cite[Theorem 3.10]{Stanley-CM-posets} showed that a vector $\mathbf{h} = (h_0, h_1, \ldots, h_d) \in \mathbb{Z}^{d+1}$ is the $h$-vector of a Cohen-Macaulay simplicial poset of rank $d$ if and only if $h_0 = 1$ and $h_j \geq 0$ for all $j$.  More generally, Novik and Swartz  \cite[Theorem 6.4]{Novik-Swartz} showed that if $P$ is a Buchsbaum simplicial poset of rank $d$, then $h'_j(P) \geq {d \choose j} \widetilde{\beta}_{j-1}(\Delta(P))$ for all $0 \leq j \leq d$.   This led them \cite[Question 7.4]{Novik-Swartz} to ask whether or not these bounds are sufficient to classify the $h$-numbers of Buchsbaum simplicial posets with prescribed Betti numbers. 

Novik and Swartz reduced this to a problem of constructing a family of Buchsbaum simplicial posets $X(k,d)$, for all $d$ and all $0 \leq k \leq d-1$, satisfying
\begin{eqnarray*}
\widetilde{\beta}_i(X(k,d)) = \begin{cases} 1, & \text{ if } i = k \\ 0, & \text{ if } i\neq k,\end{cases}
&\text{ and }& 
h'_j(X(k,d)) = \begin{cases} {d \choose j}, & \text{ if } j = 0, k+1 \\ 0, & \text{ otherwise}. \end{cases}
\end{eqnarray*}
In this paper we give an explicit construction of $X(k,d)$ for all values of $k$ and $d$.  Our construction is motivated by the study of manifold crystallizations, which encode the facet-ridge incidences of a pseudomanifold triangulation through edge-labeled multigraphs.  The simplicial posets we construct are not manifolds (or even pseudomanifolds), but they can be encoded in the same way.

The remainder of the paper is structured as follows.  In Section \ref{section:background} we provide a brief background on simplicial complexes and simplicial posets, along with a summary of Novik and Swartz's work towards providing a combinatorial classification of Buchsbaum simplicial posets.  In Section \ref{section:sect3}, we prove a series of lemmas regarding the necessary properties of the posets to be constructed.  In Section \ref{section:sect4}, we give a brief discussion of the theory of manifold crystallizations, translate it to the theory of simplicial posets, and apply it to give constructions for posets $X(k,d)$.

\section{Background and definitions} \label{section:background}

\subsection{Simplicial complexes and simplicial posets}

A \textit{simplicial complex}, $\Delta$, on vertex set $V = V(\Delta)$ is a collection of subsets $\tau \subseteq V$, called \textit{faces}, with the property that if $\tau \in \Delta$ and $\sigma \subseteq \tau$, then $\sigma\in \Delta$.  The \textit{dimension} of a face $\tau \in \Delta$ is $\dim(\tau):= |\tau|-1$, and the dimension of $\Delta$ is $\dim(\Delta):= \max\{\dim(\tau): \tau \in \Delta\}$.  The \emph{link} of a face $\tau \in \Delta$ is 
\[
\lk_{\Delta}(\tau) = \{ \gamma \in \Delta : \gamma \cap \tau = \emptyset, \gamma \cup \tau \in \Delta \}.
\]
The \textit{face poset} of a simplicial complex is the poset of faces of $\Delta$ ordered by inclusion.   The empty set is clearly the unique minimal element of the face poset of any (nonempty) simplicial complex.  Moreover, the face poset is naturally graded by $\rk(\tau) = |\tau|$; and since any face in a simplicial complex is determined by its vertices, any interval $[\sigma,\tau]$ in the face poset is a boolean lattice of rank $\rk(\tau)-\rk(\sigma)$. 

More generally, a \textit{simplicial poset}  is a poset $P$ with a unique minimal element $\hat{0}$ such that any interval $[\hat{0},\tau]$ in $P$ is a Boolean lattice.  The face poset of a simplicial complex is a simplicial poset, and hence simplicial posets serve as natural generalizations of simplicial complexes.  Many of the commonly-studied combinatorial and topological properties of simplicial complexes translate directly into corresponding properties of simplicial posets.  We refer to Stanley's book \cite{Stanley-green-book}  for further background information. 

As with simplicial complexes, a simplicial poset is naturally graded by declaring that $\rk(\tau)=r$ if $[\hat{0},\tau]$ is a Boolean lattice of rank $r$. For any simplicial poset $P$, there is a regular CW-complex $|P|$, called the \textit{geometric realization} of $P$, whose face poset is $P$.  The closed cells of $P$ are geometric simplices with every pair of cells intersecting along a (possibly empty) subcomplex of their boundaries (as opposed to a single face, as is the case for simplicial complexes); such a CW-complex is called a \emph{simplicial cell complex}.  Hereafter we will use the terminology of simplicial poset $P$ and that of its associated simplicial cell complex $|P|$ interchangeably, so a rank $r$ element of $P$ is an $(r-1)$-dimensional face, and the dimension of $P$ is $\dim(P) = \dim(|P|) = \rk(P)-1$.  In particular, following the poset terminology, we refer to the rank-one elements of a simplicial poset as \textit{atoms}; the atoms of $P$ correspond to the vertices of $|P|$.   We say $P$ is \emph{pure} if all of its maximal faces have the same dimension.

The \textit{order complex} of a simplicial poset $P$ is the simplicial complex $\Delta(P)$ whose vertices are the elements of $P - \hat{0}$, and whose faces are chains of the form $\tau_0 < \tau_1 < \cdots <\tau_r$.  Topologically, $|\Delta(P)|$ is the barycentric subdivision of $|P|$.  

The \textit{link} of a face $\sigma$ in a simplicial poset $P$ is $$\lk_{P}(\sigma) = \{ \tau \in P\; : \; \tau \geq \sigma\}.$$  It is easy to see that $\lk_{P}(\sigma)$ is a simplicial poset whose unique minimal element is $\sigma$. If $P$ is the face poset of a simplicial complex, $\lk_P(\sigma)$ is the face poset of the link of $\sigma$ in $\Delta$. Note however that in general the order complex of the link of $\sigma$ in $P$ is \emph{not} equal to the link of $\sigma$ in the order complex of $P$, i.e., $\Delta(\lk_P(\sigma)) \neq \lk_{\Delta(P)}(\sigma)$.

The most natural combinatorial invariant of a finite $(d-1)$-dimensional simplicial poset is its \textit{$f$-vector}, $f(P):=(f_{-1}(P), f_0(P), \ldots, f_{d-1}(P))$, where the \textit{$f$-numbers} $f_i(P)$ count the number of $i$-dimensional faces in $P$.  Often it is more natural to study a certain integer transformation of the $f$-vector called the \textit{$h$-vector}, $h(P) :=(h_0(P), h_1(P), \ldots, h_d(P))$ whose entries, the \textit{$h$-numbers} of $P$, are defined by the formula 
\begin{equation} \label{h-nums-formula}
h_j(P) = \sum_{i=0}^j (-1)^{j-i} {d-i \choose d-j} f_{i-1}(P).
\end{equation}
For any $(d-1)$-dimensional simplicial poset $P$, $h_0(P) = 1$  and $h_d(P) = (-1)^{d-1}\widetilde{\chi}(P)$, where $\widetilde{\chi}(P)$ denotes the reduced Euler characteristic of $P$.  Since the Euler characteristic of $P$ is inherently related to both the combinatorial and topological structure of $P$, we will also be interested in studying the \textit{(reduced) Betti numbers} of $P$ (over a field $\field$), which are defined as $\widetilde{\beta}_i(P) = \widetilde{\beta}_i(P;\field):= \dim_{\field}\widetilde{H}_i(P;\field)$.  

The primary reason for studying $h$-numbers instead of $f$-numbers is that they arise naturally when studying the \textit{face ring} of a simplicial poset.  We will not define the face ring here since the properties we are interested in studying can be defined equivalently in terms of topological information.  We refer to Stanley's book \cite{Stanley-green-book} for further information on the algebraic properties of face rings. 

We will be interested in studying two families of simplicial posets known as \textit{Cohen-Macaulay} simplicial posets and \textit{Buchsbaum} simplicial posets.  For simplicial \textit{complexes}, the properties of being Cohen-Macaulay or Buchsbaum is defined as an algebraic condition on the face ring.  Reisner \cite{Reisner} showed that the Cohen-Macaulay property for simplicial complexes is a topological condition.  Reisner's condition was later generalized to Buchsbaum complexes by Schenzel \cite{Schenzel}.  We summarize these results in the following theorem, which we will use as our definition of Cohen-Macaulay and Buchsbaum simplicial complexes. 

\begin{theorem} \label{theorem:CMLinks}
A $(d-1)$-dimensional simplicial complex $\Delta$ is Cohen-Macaulay (over a field $\field$) if and only if $$\widetilde{H}_i(\lk_{\Delta}(\tau);\field) = 0,$$ for all faces $\tau \in \Delta$ (including $\tau = \emptyset$) and all $i < \dim(\lk_{\Delta}(\tau))$.  The complex $\Delta$ is Buchsbaum (over $\field$) if and only if it is pure and the link of each of its vertices is Cohen-Macaulay (over $\field$).  
\end{theorem}

A simplicial poset $P$ is Cohen-Macaulay (respectively Buchsbaum) over $\field$ if and only if its order complex $\Delta(P)$ is Cohen-Macaulay (resp. Buchsbaum) over $\field$.  Thus any simplicial poset whose geometric realization is a $(d-1)$-sphere or ball is Cohen-Macaulay (over any field), and any simplicial poset whose realization is a manifold (with or without boundary) is Buchsbaum (over any field).     We will typically omit the field $\field$ when it is either understood or arbitrary.  The second part of Theorem \ref{theorem:CMLinks} continues to hold for Buchsbaum simplicial posets.  

\begin{lemma} \label{Lemma:CMlinks}Let $P$ be a simplicial poset. Then $P$ is Buchsbaum if and only if $\lk_P({v})$ is Cohen-Macaulay for each atom $v \in P$.

\begin{proof} By Theorem \ref{theorem:CMLinks}, $\Delta(P)$ is Buchsbaum if and only if $\lk_{\Delta(P)}(p)$ is Cohen-Macaulay for each $\hat{0} < p \in P$. On the other hand, it is clear from the definitions that $\lk_{\Delta(P)}(p) = \Delta(P_{<p}) * \Delta(\lk_P(p))$, where $*$ denotes the simplicial join. As $P$ is simplicial, $\Delta(P_{<p})$ is the barycentric subdivision of the boundary of a simplex, and thus is Cohen-Macaulay, and it is well known that $\Gamma * \Omega$ is Cohen-Macaulay if and only if both $\Gamma$ and $\Omega$ are (see e.g. Exercise 35 of \cite{Stanley-green-book}). Hence $P$ is Buchsbaum if and only if $\lk_P(p)$ is Cohen-Macaulay for each $p > \hat{0}$.

On the other hand, for $p \in P$ with $\rk(p) > 1$, there is an atom $v \in P$ with $v < p$. Then $\lk_P(p) = \lk_{\lk(v)}(p)$. Thus, as the Cohen-Macaulay property is preserved under taking links, to show that every link in $P$ is Cohen-Macaulay it suffices to check the links of atoms.

\end{proof}

\end{lemma}

The $h$-numbers of a Cohen-Macaulay simplicial poset are nonnegative \cite{Stanley-green-book} because they arise naturally as the dimensions of certain graded vector spaces associated to the face ring.  If $P$ is a Buchsbaum simplicial poset, the dimensions of the analogous vector spaces are given by the \textit{$h'$-numbers} of $P$ \cite{Novik-Swartz}, which are defined as

\begin{equation} \label{hprime-numbers}
h'_j(P) = h_j(P) + {d \choose j} \sum_{i=0}^{j-1} (-1)^{j-i-1}\widetilde{\beta}_{i-1}(P),
\end{equation}
for all $0 \leq j \leq d$. The list $(h_0'(P), h_1'(P), \ldots, h_d'(P))$ is the \emph{$h'$-vector} of $P$. Note that as $h_d(P) = (-1)^{d-1}\widetilde{\chi}(P)$, it follows that $h_d'(P) = \betti_{d-1}(P)$.  Note also that a Cohen-Macaulay simplicial complex (and thus a Cohen-Macaulay simplicial poset) may have non-vanishing reduced homology only in top degree, in which case the $h$- and $h'$-vectors coincide.

Novik and Swartz gave necessary conditions on the $h'$-vectors of Buchsbaum posets.

\begin{theorem} \cite[Theorem 6.4]{Novik-Swartz}\label{Theorem:NS} Let $P$ a Buchsbaum simplicial poset of rank $d$. Then $h'_j(P) \geq \binom{d}{j}\betti_{j-1}(P)$ for $j = 1, 2, \ldots, d-1$.

\end{theorem}

In the Cohen-Macaulay case this reduces to the condition that the $h$-numbers are non-negative. In fact, Stanley showed that this completely characterizes the possible $h$-vectors of Cohen-Macaulay simplicial posets.

\begin{theorem}\cite[Theorem 3.10]{Stanley-CM-posets}
A vector $\mathbf{h} = (h_0, h_1, \ldots, h_d) \in \mathbb{Z}^{d+1}$ is the $h$-vector of a Cohen-Macaulay simplicial poset of rank $d$ if and only if $h_0 = 1$ and $h_j \geq 0$ for all $j$.  
\end{theorem}

Thus it is natural to ask \cite[Question 7.4]{Novik-Swartz} if the condition of Theorem \ref{Theorem:NS} similarly characterizes the $h'$-vectors of Buchsbaum simplicial posets. The main result of this paper is to answer this question in the affirmative.

\begin{theorem} \label{Theorem:Goal} Let $\betti_0, \betti_1,  \ldots, \betti_{d-1}, h'_0, h'_1, \ldots, h'_{d}$ be non-negative integers. Then there is a Buchsbaum simplicial poset $P$ of rank $d$ with $h'_i(P) = h_i$ and $\betti_i(P) = \betti_i$ if and only if  $h'_0 = 1$,  $h_d' = \betti_{d-1}$ and for $j=1, 2, \ldots, d-1$, 
$h'_j \geq \binom{d}{j} \betti_{j-1}$.
\end{theorem}

\subsection{Constructing Buchsbaum posets}

To prove Theorem \ref{Theorem:Goal}, it suffices to construct, for any $\betti_i$ and $h'_i$ satisfying the conditions of the theorem, a Buchsbaum simplicial poset having those Betti numbers and $h'$-numbers. Novik and Swartz proved two lemmas that reduce this problem to certain `minimal' cases.

\begin{lemma} \label{question:classification}\cite[Lemma 7.8]{Novik-Swartz}
Let $P_1$ and $P_2$ be two disjoint $(d-1)$-dimensional Buchsbaum simplicial posets.  If $Q$ is obtained from $P_1$ and $P_2$ by identifying a facet of $P_1$ with a facet of $P_2$, then $Q$ is also a Buchsbaum poset.  Moreover, 
\begin{eqnarray*}
\widetilde{\beta}_i(Q) &=& \widetilde{\beta}_i(P_1) + \widetilde{\beta}_i(P_2), \qquad \text{for } i=0,1,\ldots, d-1, \text{ and} \\
h'_j(Q) &=& h'_j(P_1) + h'_j(P_2), \qquad \text{for } j=1,2,\ldots,d.  
\end{eqnarray*}
\end{lemma}

\begin{lemma} \cite[Lemma 7.9]{Novik-Swartz}
Let $P$ be a $(d-1)$-dimensional Buchsbaum simplicial poset, and let $g_1', \ldots, g_d'$ be nonnegative integers such that $g_j' \geq h'_j(P)$ for all $j = 1,\ldots, d$.  Then there exists a $(d-1)$-dimensional Buchsbaum simplicial poset $Q$ whose Betti numbers, except possibly for $\widetilde{\beta}_{d-1}$, coincide with those of $P$ and such that $h'_j(Q) = g'_j$ for all $1 \leq j \leq d$.
\end{lemma}

Combining these we see that  to prove Theorem \ref{Theorem:Goal} it is enough to construct a family of Buchsbaum simplicial posets $X(k,d)$ for all $d$ and all $0 \leq k \leq d-1$ such that 
\begin{eqnarray*}
\widetilde{\beta}_i(X(k,d)) = \begin{cases} 1, & \text{ if } i = k \\ 0, & \text{ if } i\neq k,\end{cases}
&\text{ and }& 
h'_j(X(k,d)) = \begin{cases} {d \choose j}, & \text{ if } j = 0, k+1 \\ 0, & \text{ otherwise}. \end{cases}
\end{eqnarray*}

For $X(0,d)$ we may take the disjoint union of two $(d-1)$-dimensional simplices, and for $X(d-1,d)$ we may take two $(d-1)$-simplices identified along their boundaries.   Novik and Swartz also gave constructions for $X(1,d)$ \cite[Lemma 7.6]{Novik-Swartz} and $X(d-2,d)$ \cite[Lemma 7.7]{Novik-Swartz} for all $d$, along with an ad-hoc construction for $X(2,5)$.  

In this paper, we provide a unified construction of $X(k,d)$ for all $d$ and all $0<k<d$. 

\begin{theorem} \label{Theorem:main}
For all $d \geq 2$ and all $0 \leq k \leq d-1$ there exists a Buchsbaum simplicial poset $X(k,d)$ with the following properties.
\begin{enumerate}
\item For all $0 \leq i \leq d-1$ and all $0 \leq j \leq d$, 
\begin{eqnarray*}
\widetilde{\beta}_i(X(k,d)) = \begin{cases} 1, & \text{ if } i = k \\ 0, & \text{ if } i\neq k,\end{cases}
&\text{ and }& 
h'_j(X(k,d)) = \begin{cases} {d \choose j}, & \text{ if } j = 0, k+1 \\ 0, & \text{ otherwise}. \end{cases}
\end{eqnarray*}
\item The link of each atom of $X(k,d)$ is shellable.
\item For each atom $v$ of $X(k,d)$,  $$h_j(\lk_{X(k,d)}(v)) = \begin{cases} {d-1 \choose j}, & \text{ if } j = 0,k \\ 0, & \text{ otherwise.} \end{cases}$$
\end{enumerate}
\end{theorem}

\section{Combinatorial analysis of $X(k,d)$}\label{section:sect3}

Our first step in proving Theorem \ref{Theorem:main} is to show that condition (1) will hold for any Buchsbaum simplicial poset having the correct $h$-vector.

\begin{lemma} \label{lemma:sufficiencies} Suppose $0<k<d$, and $P$ is a Buchsbaum simplicial poset of rank $d$ with $h_i(P) = 0$ for $0 < i \leq k$, and $h_i(P) = (-1)^{i-k+1}\binom{d}{i}$ for $k < i \leq d$. Then
\begin{eqnarray*}
\betti_i(P) = \begin{cases} 1, & \text{ if } i = k \\ 0, & \text{ if } i\neq k,\end{cases}
&\text{ and }& 
h'_j(P) = \begin{cases} {d \choose j}, & \text{ if } j = 0, k+1 \\ 0, & \text{ otherwise}. \end{cases}\end{eqnarray*}

\begin{proof}

We prove the claims simultaneously by induction.  As $P$ is non-empty, $\betti_{-1}(P) =0$, while $h_0' =1$ follows from the definitions.  Now suppose $0 < l <k$, and $\betti_j(P) = 0$ for $j <l$. It follows from the definitions that $h_{l+1}' = h_{l+1} = 0$. But by Theorem \ref{Theorem:NS}, we have
\[
\binom{d}{l+1}\betti_{l}(P)  \leq h'_{l+1}(P) = 0,
\]
and hence $\betti_l(P)= 0$. Then by induction we have $h_{l+1}'(P) = 0$ and $\betti_l(P)= 0$ for $0 \leq l <k$.  

Next, as the Betti numbers $\betti_i(P)$ vanish for $i <k$, we have $h_{k+1}'(P) = h_{k+1}(P)$, so 
\[
\binom{d}{k+1}\betti_k(P)  \leq h'_{k+1}(P) = h_{k+1}(P) = \binom{d}{k+1},
\]
and so $\betti_k(P) \leq 1$. Suppose $\betti_k(P) =0$. Then $h'_{k+2}(P) = h_{k+2}(P)$, and

\[
\binom{d}{k+2}\betti_k(P) \leq h_{k+2}(P) = -\binom{d}{k+2},
\]
a contradiction, as Betti numbers are non-negative. Hence we must have  $\betti_k(P) =1$. Finally, suppose $k < l \leq d-1$ and that  $\betti_{l'}(P) = 0$ for all $l'$ with $k \neq l' < l$. Then
\begin{align*}
\binom{d}{l+1}\betti_l(P) & \leq h'_{l+1}(P)\\
&= h_{l+1}(P) + \binom{d}{l+1}\sum_{i=0}^l (-1)^{l-i}\betti_{i-1}(P)\\
&=  h_{l+1}(P) + \binom{d}{l+1} (-1)^{l-k-1}\betti_{k}(P)\\
& = (-1)^{l+1-k+1}\binom{d}{l+1} +  (-1)^{l-k-1}\binom{d}{l+1}\\
& = 0.
\end{align*}
Hence by induction we have $\betti_l(P) = h'_{l+1}(P) = 0$ for $k<l <d$.

\end{proof}
\end{lemma}

Now to construct  $X(k,d)$ we need only build a poset that is Buchsbaum and has the correct $h$-vector. We will achieve both by controlling the links of the vertices of $X(k,d)$.

\begin{lemma} \label{Lemma:Linkhnumbers} Let $0 < k < d$ be integers and let $P$ be a simplicial poset of rank $d$ with $f_0(P) = d$.  If 
\[
h_i(\lk_P(v)) = \begin{cases} \binom{d-1}{i}, & \text{ if } i=0,k; \\ 0, & \text{ otherwise,} \end{cases}
\]
for each atom $v$ in $P$, then $h_i(P) = 0$ for $0 < i \leq k$, and $h_i(P) = (-1)^{i-k+1}\binom{d}{i}$ for $k < i \leq d$.

\begin{proof} Extending the work in \cite{Hersh-Novik-short-simplicial} on the short simplicial $h$-vector, it was shown by Swartz \cite{Swartz-lower-bounds} that if $\Delta$ is a simplicial complex of dimension $d-1$, then
\begin{align}
\sum_{v \in V} h_{i-1}(\lk_{\Delta}(v)) = ih_i(\Delta) + (d-i+1)h_{i-1}(\Delta), \label{eq-short-simplicial}
\end{align}
where $V$ is the set of vertices of $\Delta$. By viewing a simplicial poset as the face poset of a simplicial cell complex it is easily seen that identical arguments give the same relation when $\Delta$ is replaced by a simplicial poset $P$ and $V$ is taken to be its set of atoms (i.e., the vertices of the simplicial cell complex).

We now prove the claim by induction on $i$. We must have $h_0(P) = 1$, and by \eqref{eq-short-simplicial} we have 
\begin{align*}
d &= \sum_{v \in V} h_0(\lk_P(v))\\  &= h_1(P) + dh_0(P)\\ &= h_1(P) + d,
\end{align*}
so $h_1(P) = 0$. Now for $1 < i <k+1$ by induction we have
\begin{align*}
0 &= \sum_{v \in V} h_{i-1}(\lk_P(v))\\ &= ih_i(P) + (d-i+1)h_{i-1}(P) \\ &= ih_i(P),
\end{align*}
thus $h_i(P) = 0$.

Next, when $i=k$,
\begin{align*}
d\binom{d-1}{k} &= \sum_{v \in V} h_{k}(\lk_P(v))\\ &= (k+1)h_{k+1}(P) + (d-k)h_{k}(P)\\ &= (k+1)h_{k+1}(P),
\end{align*}
and therefore
\[
h_{k+1}(P) = \frac{d}{k+1}\binom{d-1}{k} = \binom{d}{k+1}.
\]

Finally, for $k+1 < i \leq d$ we have, by induction
\begin{align*}
0 &= \sum_{v \in V} h_{i-1}(\lk_P(v))\\
& = ih_i(P) + (d-i+1)h_{i-1}(P)\\
& = ih_i(P) + (d-i+1)(-1)^{i-k}\binom{d}{i-1}
\end{align*}
and so
\begin{align*}
h_i(P) & = (-1)^{k-i+1}\frac{d-i+1}{i}\binom{d}{i-1}\\
& = (-1)^{i-k+1}\binom{d}{i},
\end{align*}
as desired.

\end{proof} 

\end{lemma}

\section{Constructing $X(k,d)$}\label{section:sect4}

To construct our simplicial posets $X(k,d)$ we will adopt a graph-theoretic approach, related to the method of crystallizations of manifolds (\cite{FGG-crystallizations} is a good reference) and the graphical posets of  \cite{Murai-simplicial-cell}. However we will allow less restricted classes of graphs, as we do not require that our simplicial cell complexes are even pseudomanifolds.

\begin{definition} Let $G$ be a finite connected multigraph whose edges are labeled by colors in $[d]$. For any $S \subseteq [d]$, let $G_S$ be the restriction of $G$ to the edges whose label belongs to $S$ (and keeping all vertices of $G$). We define a poset $P(G)$ as follows: the elements of $P(G)$ are pairs $(H, S)$, where $S \subset [d]$ and $H$ is a connected component of $G_S$, ordered by $(H, S) \leq (H', S')$ if $S' \subseteq S$ and $H'$ is a subgraph of $H$.

\end{definition}

By a straightforward argument as in \cite{Murai-simplicial-cell} one may show that $P(G)$ is a simplicial poset of rank $d$ (though that paper only considers `admissible' graphs, that assumption is not needed here). It is also helpful to consider the interpretation of the graph in terms of the simplicial cell complex associated to $P(G)$.

The maximal elements of $P(G)$ are those of the form $(F, \emptyset)$, where $F$ is a single vertex of  $G$. Thus the facets of $P(G)$ correspond to the vertices of $G$, and we will from now on refer to the facet $(F, \emptyset)$ as  `the facet $F$'. The unique minimal element of $P(G)$ is $(G, [d])$, and the vertices (i.e. atoms) of $P(G)$ are pairs of the form $(H, [d] \setminus \{c\})$. We fix the coloring of the vertices of $P(G)$ so that $(H, [d] \setminus \{c\})$ receives color $c$. It follows immediately from the definitions that no face of $P(G)$ contains two vertices of the same color.  

If $F$ is a facet of $P(G)$ and $S \subseteq [d]$, let $F_S = (H, [d] \setminus S)$, where $H$ is the connected component of $G_{[d] \setminus S}$ containing $F$. Then $F_S$ corresponds to the unique face of $F$ whose vertices have the colors in $S$.

Taking all of this together, we see that the simplicial cell complex $P(G)$ is constructed from $G$ in the following way: start with a collection of $(d-1)$-simplices indexed by the vertices of $G$, and identify each of their vertex sets with $[d]$.  If two vertices $v$ and $w$ in $G$ are connected by an edge labeled $c$, identify the face of the simplex $v$ labeled $[d] \setminus \{c\}$ with that of $w$ in the natural way. Then the resulting simplicial cell complex has face poset $P(G)$. We will use these two ways of viewing $P(G)$ interchangeably. In particular, notice that if $F$ and $F'$ are facets of $P(G)$, then $F_S = F'_S$ if and only if $F$ and $F'$ are in the same connected component of $G_{[d] \setminus S}$; in other words, if there is a path from $F$ to $F'$ in $G$ that uses no edges colored by $S$.

We will often construct $G$ so that  $G_{[d] \setminus \{c\}}$ is connected for all $c \in [d]$. In this case $P(G)$ has exactly one vertex of each color, and we will identify the vertices of $P(G)$ with the elements of $[d]$.

\begin{lemma} \label{lemma:graphlinks}
 Let $G$ be an edge-colored multigraph such that $G_{[d] \setminus c}$ is connected for each $c \in [d]$. Then for each $c \in [d]$, $\lk_{P(G)}(c) = P(G_{[d] \setminus c})$.
\end{lemma} 

\begin{proof} 
Every element of $\lk_{P(G)}(c)$ is of the form $(H,S) \geq c = (G_{[d] \setminus \{c \}}, [d]\setminus \{c \})$. In other words, $S$ is a subset of $[d] \setminus \{c \}$, and $H$ is a connected component of $G_S = (G_{[d] \setminus \{c \}})_S$. So the elements of $\lk_{P(G)}(c)$  are exactly the elements of $P(G_{[d] \setminus c})$, and they have the same order relation by definition.

\end{proof}

Next we introduce a concept of shellings for posets coming from graphs. We first recall an extension of the concept of shellings for simplicial complexes to regular CW-complexes as in \cite{Bjorner-regular-cw-complexes}:

\begin{definition} \label{def:CW-shellable} Let $\Delta$ be a pure $(d-1)$-dimensional regular CW-complex. For $\sigma$ a cell of $\Delta$, let $\partial \sigma$ denote the subcomplex consisting of all the proper faces of $\sigma$. A linear ordering $\sigma_1, \sigma_2, \ldots, \sigma_r$ of the maximal cells of $\Delta$ is called a \emph{shelling} (and say we say $\Delta$ is \emph{shellable}) if either $d=1$ or $d > 1$ and the following are satisfied:
\begin{enumerate}
\item For $2 \leq j \leq r$, $ (\bigcup_{i=1}^{j-1} \partial \sigma_i) \cap \partial \sigma_j$ is pure of dimension $d-2$,
\item there is a shelling of $\partial \sigma_j$ in which the $(d-2)$-cells of $(\bigcup_{i=1}^{j-1} \partial \sigma_i) \cap \partial \sigma_j$  come first for $2 \leq j \leq r$, and
\item $\partial \sigma_1$ is shellable.
\end{enumerate}

\end{definition}

Note that by \cite{Bjorner-regular-cw-complexes}, the order complex of the face poset of a shellable CW-complex is shellable (in the usual sense) and thus is Cohen-Macaulay. Furthermore, if $\Delta$ is a simplicial cell complex, then (2) and (3) are satisfied automatically, as any ordering of the facets of the boundary of a simplex is a shelling.

\begin{definition} Let $G$ be a connected multigraph with edges labeled in $[d]$. A linear ordering, $F_1, F_2, \ldots, F_r$, of the vertices of $G$ is a \emph{graphical shelling} if  for each $i \geq 1$, the set 
\[ 
\{ S \subseteq [d] : \text{every path from $F_i$ to $F_j$ with $j<i$ uses some edge with color in $S$} \}
\]
has a unique minimal element, which we will call $R(F_i)$. Note that $R(F_1) = \emptyset$.

\end{definition}

\begin{proposition} \label{Prop:GraphicalShellings} Let $G$ be a connected multigraph with edges colored in $[d]$, and $F_1, F_2, \ldots, F_r$ a graphical shelling of $G$.  Then the corresponding ordering $F_1, F_2, \ldots, F_r$ of the facets of $P(G)$ is a shelling. Furthermore, $h_j(P) = | \{ i :  |R(F_i)| = j \}|$. 

\begin{proof} The simplicial cell complex $P(G)$ is pure of dimension $d-1$. As $P(G)$ is simplicial, we need only show property (1) from Definition \ref{def:CW-shellable}.  Suppose $1 < j \leq r$. Let $(F_j)_S$ be a  face of $(\cup_{i=1}^{j-1} \partial \sigma_i) \cap \partial \sigma_j$. Then there is some $i <j$ such that $(F_j)_S = (F_i)_S$, so there is a path in $G$ from $F_j$ to $F_i$ using no color in $S$. Thus $R(F_j) \nsubseteq S$. Pick $c \in R(F_j) \setminus S$. Then $S \subseteq [d] \setminus \{c\}$, so $(F_j)_S$ is a subface of $(F_j)_{[d] \setminus \{c\}}$, which is a $(d-2)$-face of $F_j$. But as $R(F_j) \nsubseteq [d] \setminus \{c\}$, there must be some $l < j$ such that there is a path in $G$ from $F_j$ to $F_l$ using no color in $[d] \setminus \{c\}$, and so $(F_j)_{[d] \setminus \{c\}}$ is in $(\cup_{i=1}^{j-1} \partial \sigma_i) \cap \partial \sigma_j$. Thus $(\cup_{i=1}^{j-1} \partial \sigma_i) \cap \partial \sigma_j$ is pure of dimension $d-2$, and our ordering is a shelling of $P(G)$.

The proof of the second claim is essentially the same as the proof of the corresponding result for shellable simplicial complexes, and follows by straightforward computation from the observation that at each step of the shelling, the faces of $F_j$ that are not in any earlier facet are those of the form $(F_j)_S$ such that $R(F_j) \subseteq S \subseteq [d]$.

\end{proof}
\end{proposition}

We are now ready to define the simplicial posets $X(k,d)$.  The graphs we will use to realize $X(k,d)$ will be defined in terms of binary words; the idea is somewhat reminiscent of the construction of $B(i,d)$ in \cite{MR2854181}.  Since $X(0,d)$ is the disjoint union of two $(d-1)$-simplices, we will only consider the simplicial posets $X(k,d)$ with $0 < k < d$ for the remainder of this paper.

\begin{definition} Suppose $k, d$ are positive integers with $k<d$. Let $W_d$ be the set of words on alphabet $\{0,1\}$ of length $d$ and first letter $1$. For $w = w_1w_2 \ldots w_d$ in $W_d$, we may write $w$ as a concatenation of subwords $w = B_1(w)B_2(w) \ldots B_q(w)$ where $B_i(w)$ is a word of all $1$'s if $i$ is odd and all $0$'s if $i$ is even. We will call these the \emph{blocks} of $w$. For any $1 \leq j \leq d$, let $b_j(w)$ be the index of the block of $w$ containing $w_j$.  Let $W_d(k)$ denote the set of words in $W_d$ with exactly $k+1$ blocks.
\end{definition}
\begin{definition}
For $w = w_1w_2 \ldots w_d \in W_d$ let $REP(w) = \{ i>1: w_i \neq w_{i-1}\}$ be the set of \emph{right endpoints} of $w$, and   $LEP(w) = \{ i<d: w_i \neq w_{i+1}\}$ be the set of \emph{left endpoints} of $w$. 
\end{definition}
The terminology refers to the fact that $LEP$ and $REP$ contain the indices of the first and last elements respectively of the blocks of $w$.  Note that we do not include $1$ as a left endpoint or $d$ as a right endpoint. However we allow $1$ to be a right endpoint (if the first block has size one), and $d$ to  be a left endpoint (if the last block has size one).

\begin{definition} For $k, d$ positive integers, $k<d$, let $G'(k,d)$ be the edge-labeled multigraph on vertex set $W_d(k)$ such that if two words differ only in position $j$ they are connected by an edge labeled $j$.

Let $G(k,d)$ be the graph obtained by adding a new vertex, $\alpha$, to $G'(k,d)$, and connecting a vertex $w=w_1w_2 \ldots w_d$ of $G'(k,d)$  to $\alpha$ by an edge labeled $j$ if $w_j$ is contained in a block of size one in $w$.

\end{definition}

\begin{example}
The vertices of $G'(1,d)$ correspond to binary words whose first letter is 1 that have two blocks.  We label the vertices as 

\begin{eqnarray*}
v_1 &=& 100\cdots 00  \\
v_2 &=& 110\cdots 00 \\
v_3 &=& 111\cdots 00 \\
&\vdots& \\
v_{d-1} &=& 111\cdots 10.
\end{eqnarray*}

Thus the graph $G(1,d)$ is
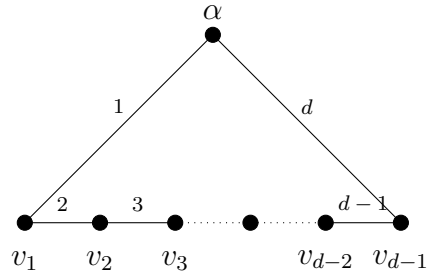
\begin{center}
\begin{tikzpicture}
\draw (0,0) -- (2,0);
\draw[dotted] (2,0) -- (4,0);
\draw(4,0) -- (5,0);
\draw[fill=black] (0,0) circle (.1);
\draw[fill=black] (1,0) circle (.1);
\draw[fill=black] (2,0) circle (.1);
\draw[fill=black] (3,0) circle (.1);
\draw[fill=black] (4,0) circle (.1);
\draw[fill=black] (5,0) circle (.1);
\draw (0,-.5) node {$v_1$};
\draw (1,-.5) node {$v_2$};
\draw (2,-.5) node {$v_3$};
\draw (4,-.5) node {$v_{d-2}$};
\draw (5,-.5) node {$v_{d-1}$};
\draw (.5,.25) node {\tiny{$2$}};
\draw (1.5,.25) node {\tiny{$3$}};
\draw[fill=black] (2.5,2.5) circle (.1);
\draw (0,0) -- (2.5,2.5) -- (5,0);
\draw (4.5,.25) node {\tiny{$d-1$}};
\draw (2.5,2.8) node {$\alpha$};
\draw (1.25,1.55) node {\tiny$1$};
\draw (3.75,1.55) node {\tiny$d$};
\end{tikzpicture}
\end{center}

The corresponding simplicial poset $P(G(1,d))$ has $d$ facets, $F_0,F_1, F_2, \ldots, F_{d-1}$, where $F_i$ intersects $F_{i+1}$ along the facet opposite vertex $i+1$ for all $1 \leq i \leq d$ (where addition is taken modulo $d$).  In this picture, $F_0$ corresponds to vertex $\alpha$, and $F_j$ corresponds to $v_j$ for all other $j$.  This matches the complex $X(1,d)$ constructed by Novik and Swartz \cite[Lemma 7.6]{Novik-Swartz}.  

For concreteness, the geometric realization of $P(G(1,3))$ is shown below with its underlying graph shown in gray. 

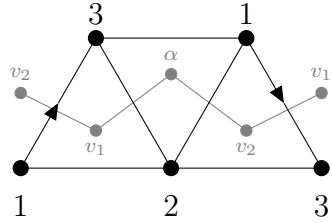
\begin{center}
\begin{tikzpicture}[>=triangle 45]
\draw[fill=gray, color=gray] (2,1.25) circle (.075);
\draw[color=gray] (2,1.5) node {\tiny$\alpha$};
\draw[fill=gray, color=gray] (1,.5) circle (.075);
\draw[color=gray] (1,.25) node {\tiny$v_1$};
\draw[fill=gray, color=gray] (3,.5) circle (.075);
\draw[color=gray] (3,.25) node {\tiny$v_2$};
\draw[color=gray] (0,1) -- (1,.5) -- (2,1.25) -- (3,.5) -- (4,1);
\draw[color=gray, fill=gray] (0,1) circle (.075);
\draw[color=gray, fill=gray] (4,1) circle (.075);
\draw[color=gray] (0,1.25) node {\tiny$v_2$};
\draw[color=gray] (4,1.25) node {\tiny$v_1$};

\draw (0,0) -- (2,0) -- (1,1.73);
\draw (1,1.73) -- (3,1.73) -- (2,0);
\draw (2,0) -- (4,0);
\draw[->] (0,0) -- (.5,.866);
\draw (.5,.866) -- (1,1.73);
\draw [->](3,1.73) -- (3.5,1.73-.866);
\draw (3.5,1.73-.866) -- (4,0);
\draw[fill=black] (0,0) circle (.1);
\draw[fill=black] (2,0) circle (.1);
\draw[fill=black] (4,0) circle (.1);
\draw[fill=black] (1,1.73) circle (.1);
\draw[fill=black] (3,1.73) circle (.1);
\draw (0,-.5) node {$1$};
\draw (2,-.5) node {$2$};
\draw (1,2.05) node {$3$};
\draw (4,-.5) node {$3$};
\draw (3,2.05) node {$1$};
\end{tikzpicture}
\end{center}
The simplicial cell complex $X(1,d)$ is either a trivial disc bundle (when $d$ is even) or a twisted disc bundle (when $d$ is odd) over $\mathbb{S}^1$.  In general, the complexes $X(k,d)$ are not pseudomanifolds, which we discuss further in Section \ref{section:manifolds}. 
\end{example}

\begin{example}
We will also illustrate $G(2,5)$.  The vertices of $G'(2,5)$ correspond to binary words whose first letter is 1 that have three blocks.  The vertices in this case are
\begin{align*}
v_{12} = & 10111 & v_{23}= & 11011 \\
v_{13} = & 10011 & v_{24}= & 11001 \\
v_{14} = & 10001 & v_{34}= & 11101.  
\end{align*}
The graph $G(2,5)$ is shown below.  We have drawn the vertex $\alpha$ three times to make the drawing planar, and we have made this vertex white in order to distinguish it from the vertices of $G'(2,5)$.
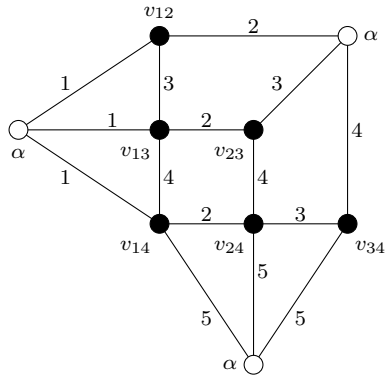
\begin{center}
\begin{tikzpicture}[scale=1.25]
\draw (0,0) -- (2,0);
\draw (0,0) -- (0,2);
\draw (1,0) -- (1,1) -- (0,1);
\draw[fill=black] (0,0) circle (.1);
\draw[fill=black] (0,1) circle (.1);
\draw[fill=black] (0,2) circle (.1);
\draw[fill=black] (1,1) circle (.1);
\draw[fill=black] (1,0) circle (.1);
\draw[fill=black] (2,0) circle (.1);
\draw (0,2.25) node {\tiny$v_{12}$};
\draw (.75,.75) node {\tiny$v_{23}$};
\draw (-.25,.75) node {\tiny$v_{13}$};
\draw (-.25,-.25) node {\tiny$v_{14}$};
\draw (.75,-.25) node {\tiny$v_{24}$};
\draw (2.25,-.25) node {\tiny$v_{34}$};
\draw (.1,1.5) node {\tiny$3$};
\draw (.1,.5) node {\tiny$4$};
\draw (1.1,.5) node {\tiny$4$};
\draw (.5,.1) node {\tiny$2$};
\draw (.5,1.1) node {\tiny$2$};
\draw (1.5,.1) node {\tiny$3$};
\draw (0,2) -- (2,2) -- (2,0);
\draw (1,1) -- (2,2);
\draw (2.25,2) node {\tiny$\alpha$};
\draw (1,2.1) node {\tiny$2$};
\draw (1.25,1.5) node {\tiny$3$};
\draw (2.1,1) node {\tiny$4$};
\draw(.75,-1.5) node {\tiny$\alpha$};
\draw (0,0) -- (1,-1.5) -- (2,0);
\draw (1,0) -- (1,-1.5);
\draw (.5,-1) node {\tiny$5$};
\draw (1.5,-1) node {\tiny$5$};
\draw (1.1,-.5) node {\tiny$5$};
\draw (-1.5,.75) node {\tiny$\alpha$};
\draw (0,0) -- (-1.5,1) -- (0,2);
\draw (-1.5,1) -- (0,1);
\draw (-1,.5) node {\tiny$1$};
\draw (-1,1.5) node {\tiny$1$};
\draw (-.5,1.1) node {\tiny$1$};
\draw[fill=white] (2,2) circle (.1);
\draw[fill=white] (1,-1.5) circle (.1); 
\draw[fill=white] (-1.5,1) circle (.1);
\end{tikzpicture}
\end{center}
\end{example}

Traversing an edge labeled $j$ in $G'(k,d)$ from $w$ to $w'$ corresponds to changing the value of $w_j$ from $1$ to $0$ or $0$ to $1$; we will refer to such a move as a `flip' in position $j$. A path in $G'(k,d)$ from $w$ to $w'$ corresponds to performing a series of successive flips that take $w$ to $w'$, with each word along the way being an element of $W_d(k)$. A path from $w$ to $\alpha$ in $G(k,d)$ corresponds to performing a series of flips such that the last flip reduces the number of blocks.

\begin{lemma} \label{Lemma:blockchange} Any path in $G(k,d)$ from $w$ to $w'$ with $b_c(w) \neq b_c(w')$ must either traverse an edge colored $c$ or visit $\alpha$.

\begin{proof}
 Clearly if $u, v \in W_d(k)$ are connected by an edge labeled $i$ in $G(k,d)$, then either $i = c$ or $b_c(w) = b_c(w')$. 
\end{proof}
\end{lemma}

\begin{lemma} For positive integers $k < d$, $G(k,d)$ has $\binom{d-1}{k} + 1$ vertices. Furthermore, $G(k,d)_{[d] \setminus \{c \}}$ is connected for each $c \in [d]$.

\begin{proof} Clearly the elements of $W_d(k)$ are in bijection with the set of compositions of $d$ into $k+1$ parts, of which there are $\binom{d-1}{k}$, and the first statement follows.

To prove the second claim, let $w \in W_d(k)$ and pick any block $B_i(w) = w_r \ldots w_q$ not containing $w_c$ (there are at least two blocks so this is always possible). Then the flips in positions $r, r+1, \ldots, q$ (if we have chosen the first block, reverse the order) give a path in $G(k,d)_{[d] \setminus \{c\}}$ from $w$ to $\alpha$. Thus every vertex of $G(k, d)_{[d] \setminus \{c\}}$ is joined by a path to $\alpha$, so the graph is connected.

\end{proof}

\end{lemma}

\begin{theorem} \label{Theorem:Mainshelling} Fix positive integers $k<d$, and $G = G(k,d)$.Then for each $c \in [d]$, $G_{[d] \setminus \{c \}}$ admits a graphical shelling, $\alpha, w^1, w^2, \ldots w^{\binom{d-1}{k}}$, with $|R(w^j)| = k$ for each $j$.

\end{theorem}

\begin{corollary} For $k <d$ positive integers, Theorem \ref{Theorem:main} is satisfied by letting $X(k,d) = P(G(k,d))$.

\begin{proof} By Theorem \ref{Theorem:Mainshelling}, Proposition \ref{Prop:GraphicalShellings}, and Lemma \ref{lemma:graphlinks} we see that the link of each vertex of $P(G(k,d))$ is shellable and has $h$-vector $\left(1, 0, \ldots 0, \binom{d-1}{k}, 0 \ldots, 0\right)$, where the $\binom{d-1}{k}$ entry is in the $k^{th}$ position. Thus, by  Lemma \ref{Lemma:CMlinks}, $P(G(k,d))$ is Buchsbaum, and by Lemmas \ref{Lemma:Linkhnumbers} and \ref{lemma:sufficiencies} we see that $P(G(k,d))$ has the desired $h'$-numbers and Betti numbers.

\end{proof}

\end{corollary}

\begin{proof}[Proof of Theorem \ref{Theorem:Mainshelling}] Fix $c \in [d]$. We define a total order on the vertices of $G_{[d] \setminus \{c \}}$. First, set $\alpha < w$ for all $w \in W_d(k)$. Next, for $w \in W_d(k)$, let $I(w) = REP(w) \cap \{1, 2, \ldots c-1\}$ and $E(w) = LEP(w) \cap \{c+1, \ldots d\}$. Suppose $w, w' \in W_d(k)$, with $I(w) = \{a_1 < a_2 < \ldots < a_r \}$, $E(w) = \{ c_1 < c_2 < \ldots < c_s\}$, $I(w') = \{a'_1 < a'_2 < \ldots < a'_{r'} \}$, and $E(w') = \{ c'_1 < c'_2 < \ldots < c'_{s'}\}$. Say $w < w'$ if either
\begin{enumerate}
\item $|I(w)| < |I(w')|$,
\item $|I(w)| = |I(w')|$, and there is an index $j$ such that $a_i = a_i'$ for $i <j$, and $a_j < a_j'$, or
\item $I(w) = I(w')$ and there is an index $j$ such that $c_i = c_i$ for $i >j$ and $c_j > c_j'$.
\end{enumerate}
(Note that this is the lexicographic order on $I(w)$ if we reorder the indices $1, 2, \ldots c, d, d-1, \ldots c+1$. We adopt this more explicit definition to clarify the argument that follows.)

Notice that for $b_c(w) = j$, there are $j$ blocks in the word $w_1w_2 \ldots w_c$, and so $|I(w)| = j-1$, while there are $k+1-j+1$ blocks in $w_cw_{c+1} \ldots w_d$, so $|E(w)| = k-j+1$. In particular, $|I(w) \cup E(w)| = k$.

We claim that the ordering $\alpha < w^1 < w^2 < \cdots < w^{d-1 \choose k}$ is a shelling of $G_{[d] \setminus \{c \}}$ with $R(w^j) = I(w) \cup E(w)$ for all $j \geq 1$.

Suppose $S \subseteq [d]$ and  $I(w) \cup E(w) \nsubseteq S$. We must show that there is a path from $w$ to an earlier vertex that uses no color in $S$. Choose $i \in I(w) \cup E(w) \setminus S$. If $w_i$ is the only element in its block, then there is an edge from $w$ to $\alpha$ colored $i$, and we are done. Otherwise, there is an edge colored $i$ from $w$ to $w' \in W_k(d)$. If $i \in I(w) = \{a_1 < a_2 < \cdots < i < \cdots a_r \}$, then $I(w') =  \{a_1 < a_2 < \cdots < i-1 < \cdots a_r \}$, and $w' < w$. Similarly, if $i \in E(w) = \{c_1 < c_2 < \cdots < i < \cdots c_s \}$, then $E(w') =  \{c_1 < c_2 < \cdots < i+1 < \cdots c_r \}$, and again $w' < w$.

It remains to show that for $u < w$, every path from $w$ to $u$ uses some edge in $I(w) \cup E(w)$. First, suppose we have a path from $w$ to $\alpha$. Let $i$ be the color of the final edge in the path. In particular, our path must flip every element of the block of $w$ containing $w_i$. Thus $w_c$ is not in this block, so if it is the first block, the right-endpoint of the block is in $I(w)$, and if it is the last block the left-endpoint must be in $E(w)$. In the case that $w_j$ is in neither the first nor last block of $w$, then the path must flip both the left- and right-endpoints of the block. In any case our path traverses an edge with color in $I(w) \cup E(w)$.

Now suppose we have a path from $w$ to $w' < w$, $w' \in W_d(k)$ (and assume $w'$ is the first vertex less than $w$ reached by the path, so in particular the path does not visit $\alpha$). Again let $I(w) = \{a_1 < a_2 < \ldots a_r \}$, $E(w) = \{ c_1 < c_2 < \ldots c_s\}$, $I(w') = \{a'_1 < a'_2 < \ldots a'_{r'} \}$, and $E(w') = \{ c'_1 < c'_2 < \ldots c'_{s}\}$. We have three cases:

(1) Suppose $|I(w)| < |I(w')|$. Then $b_c(w) \neq b_c(w')$, and thus as our path cannot use color $c$, by Lemma \ref{Lemma:blockchange} it must pass through $\alpha$, a contradiction.

(2) Suppose $|I(w)| = |I(w')|$, and there is an index $j$ such that $a_i = a_i'$ for $i <j$, and $a_j > a_j'$. Then $b_{a_j}(w) < b_{a_j}(w')$, so by Lemma \ref{Lemma:blockchange}, the path must use an edge of color $a_j \in I(w)$.

(3) Finally, suppose  $I(w) = I(w')$ and there is an index $j$ such that $c_i = c_i$ for $i >j$ and $c_j < c_j'$. Then $b_{c_j}(w) > b_{c_j}(w)$, and our path must use an edge of color $c_j \in E(w)$.
\end{proof}

\section{Manifolds with Boundary} \label{section:manifolds}
To conclude, we consider the case of simplicial posets whose geometric realizations are manifolds with boundary (the face vectors of simplicial posets realized as odd-dimensional manifolds without boundary were characterized by Murai in \cite{Murai-simplicial-cell}). In particular we note that the conditions of Theorem \ref{Theorem:Goal} are insufficient to give a characterization in this case. 

In particular, note that the simplicial posets $X(k,d)$ constructed for Theorem \ref{Theorem:main} are not in general manifolds, or even homology manifolds or pseudomanifolds (with boundary), outside the cases $k=1, d-1$. In fact it is easy to check that these simplicial posets contain faces of codimension one that are contained in more than two facets, which is of course impossible in a (homology/pseudo-) manifold with boundary. It turns out that when $\frac{d}{2} \leq k < d-1$, this is impossible to avoid. Before proving this we recall two standard binomial coefficient identities.

\begin{lemma} \label{lemma:binomial-coeffs}
 Let $0 \leq k \leq d-1$. Then
\begin{align*}
 \sum_{j=k+1}^d (-1)^{j-k+1}\binom{d}{j} & = \binom{d-1}{k}.
 \end{align*}
 For $0 \leq k < d-1$,
 \begin{align*}
 \sum_{j=k+1}^{d-1} (-1)^{j-k+1}(d-j)\binom{d}{j} & = d\binom{d-2}{k}.
\end{align*}
%
%
%

\end{lemma}

\begin{theorem} Let $P$ be a $(d-1)$-dimensional Buchsbaum simplicial poset, $d>4$, satisfying condition (1) of Theorem \ref{Theorem:main} for some $k$, $\frac{d}{2} \leq k < d-1$. Then $P$ has a $(d-2)$-face which is contained in at least three $(d-1)$-faces.

\begin{proof} Suppose to the contrary the each $(d-2)$-face of $P$ is contained in at most two facets. Let $A$ and $B$ be the number of $(d-2)$-faces of $P$ contained in exactly one or two facets of $P$, respectively, so $A+B = f_{d-2}(P)$. Furthermore, as each facet of $P$ contains exactly $d$-many $(d-2)$-faces, we have $A + 2B = df_{d-1}(P)$. Solving we see that $A = 2f_{d-2} - df_{d-1}$, and in particular it follows that we must have $2f_{d-2} - df_{d-1} \geq 0$.

Now, from the definitions and property (1), we have $h_0(P) =1$, $h_i(P) = 0$ for $0 < i \leq k$, and $h_i(P) = (-1)^{i-k+1}\binom{d}{i}$ for $k < i \leq d$. Recall the well-known formula for recovering $f$-numbers from $h$-numbers,
\[
f_{i-1} = \sum_{j = 0}^i\binom{d-j}{d-i}h_j.
\]

In particular, it follows from the identities in Lemma \ref{lemma:binomial-coeffs} that
\begin{align*}
f_{d-1}(h) 
& = 1 + \binom{d-1}{k}
\end{align*}
and
\begin{align*}
f_{d-2}(h) 
& = d + d\binom{d-2}{k}.
\end{align*}

Thus we have
\begin{align*}
0 & \leq 2\left(d + d\binom{d-2}{k}\right) -  d\left(1 + \binom{d-1}{k}\right),\\
\end{align*}
so
\begin{align*}
0 & \leq  1 + 2\binom{d-2}{k} -  \binom{d-1}{k},\\
& = 1 + \binom{d-2}{k} - \binom{d-2}{k-1}\\
& = 1 + \frac{(d-1) - 2k}{d-1}\binom{d-1}{k}\\
& \leq 1 -\frac{1}{d-1}\binom{d-1}{k},\\
\end{align*}
with equality only if $k = \frac{d}{2}$. On the other hand, $\binom{d-1}{k} \geq d-1$ when $\frac{d}{2} \leq k < d-1$, with equality only when $k = d-2$. Combining these we have 
\[
0  \leq 1 -\frac{1}{d-1}\binom{d-1}{k} \leq 0,
\]
and so must have equality in both cases; thus $\frac{d}{2} = k = d-2$, and $d = 4$, contradicting our assumption.

\end{proof}

\end{theorem}

In particular when $\frac{d}{2} \leq k < d-1$ it is impossible to build simplicial posets satisfying the condition (1) of Theorem \ref{Theorem:main} that are manifolds, so the bounds given by Novik and Swartz are not sharp for manifolds with boundary. This leads one to ask what the right characterization is in this case. At this time we have no good conjecture. 

It also may be possible that the posets $X(k,d)$ of Theorem \ref{Theorem:main} could be built so as to be manifolds for more pairs $k,d$ than are presented here. For example, in addition to the $k=1$ or $d-1$ cases, recall that the (different) construction of $X(2,5)$ in \cite{Novik-Swartz} is obtained by removing a single facet from a simplicial cell decomposition of $\mathbb{C}P^2$, and is thus a manifold with boundary. More progress on these minimal cases may yet shed light on the general problem of manifolds with boundary.

\section{Acknowledgements} 
The authors thank Isabella Novik for several helpful discussions and her many useful comments on an earlier draft of this paper, and Ed Swartz for his comments, which inspired the inclusion of section \ref{section:manifolds}.

\end{document}